\documentclass[12pt]{article}

\usepackage{graphicx} 
\usepackage{amsmath,amssymb, MnSymbol, amsthm}

\usepackage{pgf}
\usepackage{amsfonts}
\usepackage{amsmath}
\usepackage{amssymb}
\usepackage{tikz}
\usepackage{stmaryrd}
\usetikzlibrary{automata,trees,arrows}
\usepgflibrary{shapes.geometric}
\usepackage{amsfonts}

\usepackage{graphicx}

\newtheorem{thm}{Theorem}[section]
\newtheorem{prop}[thm]{Proposition}
\newtheorem{cor}[thm]{Corollary}

\newtheorem{lem}[thm]{Lemma}
\newtheorem{dfn}[thm]{Definition}

\newtheorem{obs}[thm]{Observation}
\newcommand{\fl}{\mbox{$\mathfrak{F}$}}
\newcommand{\M}{\mbox{$\mathfrak{M}$}}

\begin{document}

\title{Fragments of Some Subintuitionistic Logics}

\author{\textbf{Fatemeh Shirmohammadzadeh Maleki}\\
Department of Logic, Iranian Institute of Philosophy\\
Arakelian 4, Vali-e-Asr, Tehran, Iran,~f.shmaleki2012@yahoo.com\\
\textbf{Dick de Jongh}\\
Institute for Logic, Language and Computation, University of Amsterdam\\
The Netherlands, D.H.J.deJongh@uva.nl}

\maketitle
\date{}
\begin{abstract} In this article we determine the implicational fragments of most of the known subintuitionistic logics.
\end{abstract}

\section{Introduction}

In a series of papers \cite{Dic,  DF2, Dic4, FD6, FD} we discussed a number of subintuitionistic logics between {\sf WF} and {\sf IPC}. 
Logics between {\sf F} and {\sf IPC} had been discussed before by others and by us as well \cite{Celani2, Co, Dic, DF2, a2}.
As {\sf IPC} the language of subintuitionistic logics contains the connectives $ \vee, \wedge, \rightarrow $  and the propositional constant $  \perp$.

In this paper we investigate the implicational fragments of these logics. We denote a fragment by listing the connectives between square brackets, so $ [\rightarrow] $ is the fragment consisting of the formulas that only contain the connective  $ \rightarrow $, the $ \rightarrow $-formulas. For any subintuitionistic logic $L$ we define $L_{[\rightarrow]} $ as the logic consisting of  the $ \rightarrow $-formulas provable in $ L$. Our main object is to establish an axiomatization of $L_{[\rightarrow]} $. The study of the $[\rightarrow]$-fragments creates an overview of implicational principles weaker than the positive implication of {\sf IPC}, and their relative strengths.
 
We have an additional interest in $ L_{[\rightarrow, \wedge]} $,  the $[\rightarrow , \wedge]$-fragment of $L$.
This fragment is in {\sf IPC} closely related to the  $[\rightarrow]$-fragment~\cite{Ren}, and for many logics easier to describe. The subject of $ [\rightarrow] $-fragments of subintuitionistic logics was previously studied by Do\v{s}en~\cite{Dosen} and Kikuchi~\cite{Kent}. Also relevant is~\cite{Maleki}, which presents cut-free sequent systems for {\sf WF} and some of its extensions. In these systems, the rules governing the implication connective $\rightarrow$ alone yield fragments that are equivalent to the Hilbert-style axiomatizations developed in the present paper.
The logics studied in this paper including their models and the axioms used to axiomatize them are described below in two Appendices, the logics with Kripke models in~\ref{kripke}, the ones with neighborhood models in~\ref{neighbor}.

In Section 2 we describe $\mathsf{F_{[\rightarrow]}}$ and $\mathsf{F_{[\rightarrow,\wedge]}}$. In Section~\ref{extf} we move to extensions of {\sf F}. Then, In Section~\ref{fwf}, we treat the basic logic with neighborhood models {\sf WF}. The implicational fragments of its extensions are studied in Section~\ref{extwf}. In taht section we also discuss some unresolved questions, and state soem conjectures about them. The implication-conjunction fragments of some extensions of {\sf WF} in are discussed in Section~\ref{icextwf}.

\section{Fragments of the subintuitionistic Logic {\sf F}}

We discuss in this section the implicational and the implication-conjunction fragment of {\sf F}, the basic subintuitionistic logic with Kripke models \cite{Co,Dic}.

\subsection{The implicational fragment of $ {\sf F}  $}


The Kripke frames  of the fragments are simply relativized versions of the Kripke frames  of the fragments of \cite{Dic, Co}. As mentioned above we describe the properties of the logic {\sf F}, its extensions and their models in Appendix~\ref{kripke}.

\begin{dfn}\label{Frr}
 A \textbf{rooted subintuitionistic $[\rightarrow]$-Kripke model} is a quadruple $ \langle W, g, R, V \rangle $, where  $ \langle W, g, R \rangle $ is a rooted subintuitionistic Kripke frame and $ V:P\rightarrow 2^{W} $  a valuation function on the set of propositional variables $P$. The binary relation $ \Vdash $ is defined on $ w \in W $ as follows.
\begin{enumerate}
\item $ w \Vdash p~~~~~~~~\Leftrightarrow ~ ~w \in V(p) $, for any $ p \in P $,



\item  $w \Vdash A\rightarrow B  ~\Leftrightarrow$ ~ for each v with $w R v$, if $v \Vdash A$ then $v \Vdash B$.
\end{enumerate}
We define $\Gamma\Vdash_{{\sf F_{[\rightarrow]}} }\! A$ iff for all $ \mathfrak{M}$, $ w \in \mathfrak{M}$, if $\mathfrak{M}, w\Vdash \Gamma  $ then $\mathfrak{M}, w\Vdash A  $.
\end{dfn}

We prove strong completeness for a Hilbert system for $ {\sf F_{[\rightarrow]}}  $. Simple completeness of the fragment was already proved by K. Do\v{s}en \cite{Dosen}. Here we give a new proof.


\begin{thm}\label{CFrr}
Assume $ \bar{A_{n}}\rightarrow B $ stands for B if  $ n=0 $, for $ A_{1}\rightarrow B $ if $ n=1 $, and for $A_{1}\rightarrow (A_{2}\rightarrow  ... \rightarrow (A_{n}\rightarrow B) ... )  $ if $ n\geq 2 $. Then the following  axiom schemes and rules in $ \mathcal{L} _{\rightarrow} $ axiomatize the fragment $ {\sf F_{[\rightarrow]}}  $:
\begin{enumerate}
\item $A\rightarrow A$
\item $ \dfrac{A ~~A\rightarrow B}{B}$
\item $  \dfrac{A}{B\rightarrow A} $
\item $ \dfrac{\bar{A_{n}}\rightarrow (B\rightarrow C)~~~~\bar{A_{n}}\rightarrow (C\rightarrow D)}{\bar{A_{n}}\rightarrow (B\rightarrow D)} ~~~~n\geq 0$.
\end{enumerate}
\end{thm}
  Before the proof we give some definitions. We will immediately use $ \mathsf{ F_{[\rightarrow]}}  $ for the logic based on this axiomatic system or the set of formulas provable in the logic, making sure to avoid confusion with the fragment $ \mathsf{ F_{[\rightarrow]}}  $. 

\begin{dfn}

We define $\Gamma\vdash_{{\sf F_{[\rightarrow]}} }\!A$ iff there is a  derivation of A from $ \Gamma $ using rule 3 of the system $ {\sf F_{[\rightarrow]}}  $ only when there are no assumptions, and rule 2, $M\!P$, only when the derivation of $ A\rightarrow B $ contains no assumptions. We call this use of $M\!P$ \textbf{weak $M\!P$}.
\end{dfn}

\begin{thm}\label{lli}
\em\textbf{(Weak Deduction Theorem)} $A\vdash_{\mathsf{F_{[\rightarrow]}} }\!B  $ if and only if $\vdash_{{\sf F_{[\rightarrow]}} }\! A\rightarrow B $.
\end{thm}
\begin{proof} $ \Rightarrow $: By induction on the length of the proof of $A\vdash_{{\sf F_{[\rightarrow]}} }\!B  $.

\noindent If $B$ is a theorem of $ {\sf F_{[\rightarrow]}}  $, then $ \vdash_{{\sf F_{[\rightarrow]}} }\!B $. So,
by rule 3, $\vdash_{{\sf F_{[\rightarrow]}} }\!A\rightarrow B  $.

\noindent The case $A\vdash_{{\sf F_{[\rightarrow]}} }\!A  $ is covered by Axiom 1.



\noindent If $ A\vdash_{{\sf F_{[\rightarrow]}} }\!B   $, and $\vdash_{{\sf F_{[\rightarrow]}} }\!B\rightarrow C  $ was used to obtain $ A\vdash_{{\sf F_{[\rightarrow]}} }\!C   $, then, by the induction hypothesis, $\vdash_{{\sf F_{[\rightarrow]}} }\!A\rightarrow B  $ and hence, by rule 4,  $ \vdash_{{\sf F_{[\rightarrow]}} }\!A\rightarrow C $.

\noindent  If $ A\vdash_{{\sf F_{[\rightarrow]}} }\!\bar{A_{n}}\rightarrow (B\rightarrow C)   $ and $ A\vdash_{{\sf F_{[\rightarrow]}} }\!\bar{A_{n}}\rightarrow (C\rightarrow D)   $.  Then, by the induction hypothesis $\vdash_{{\sf F_{[\rightarrow]}} }\!A\rightarrow (\bar{A_{n}}\rightarrow (B\rightarrow C) )  $ and $\vdash_{{\sf F_{[\rightarrow]}} }\!A\rightarrow (\bar{A_{n}}\rightarrow (C\rightarrow D)  ) $.  Now, by  rule 4, we conclude that $\vdash_{{\sf F_{[\rightarrow]}} }\!A\rightarrow (\bar{A_{n}}\rightarrow (B\rightarrow D)  ) $.

\noindent $ \Leftarrow $: Immediate by weak $M\!P$.
\end{proof}

\begin{dfn}\label{r1}
\begin{enumerate}
\item  A set of sentences $ \Delta$ is an $\mathsf{F_{[\rightarrow]} }$-theory if and only if
\begin{enumerate}
\item   $\bar{A_{n}}\rightarrow (B\rightarrow C) \in \Delta, ~\bar{A_{n}}\rightarrow (C\rightarrow D) \in \Delta ~\Rightarrow~ \bar{A_{n}}\rightarrow (B\rightarrow D) \in \Delta, $

\item  $ \vdash_{\sf F_{[\rightarrow]}}A\rightarrow B ~\Rightarrow$  $($if $ A \in \Delta$, then $ B \in \Delta), $

\item ${\sf F_{[\rightarrow]}}$ is contained in $ \Delta $.
\end{enumerate}
\item  For theories $\Gamma,\,\Delta$, $ \Gamma R\,\Delta$ iff for all $A\rightarrow B \in \Gamma$, $A\in \Delta\,\Rightarrow B \in \Delta$.
\end{enumerate}
\end{dfn}

We may call $\mathsf{F_{[\rightarrow]} }$-theories just $[\rightarrow]$-theories, or even theories, for short. The $R$-relation is just sued in the definition of canonical models.

\begin{prop}
$ \Delta $ is an $\mathsf{F_{[\rightarrow]} }$-theory  $ \Longleftrightarrow $ $( \Delta\vdash_{{\sf F_{[\rightarrow]}} }\!A $ if and only if $ A \in \Delta)$.
\end{prop}
\begin{proof}
From right to left is trivial. The other direction is by an easy induction on the length of the proof.
\end{proof}The next few propositions are standard, and have also been produced for the whole of {\sf F} in~\cite{Dic}. Here we will give complete proofs, but when discussing later fragments in this paper we mainly just restrict ourselves to comments about what has to be changed or added.

\begin{dfn}\label{r2}
We call $ \left\lbrace  A\,|\vdash_{{\sf F_{[\rightarrow]}} }\!A\right\rbrace$ the \textbf{empty $ {[\rightarrow]} $-theory}.
\end{dfn}
\begin{prop}\label{empty}
The empty $ {[\rightarrow]} $-theory $\Delta$ is a $ {[\rightarrow]} $-theory.
\end{prop}

\begin{proof}
Let $\bar{A_{n}}\rightarrow (B\rightarrow C) \in \Delta$ and $\bar{A_{n}}\rightarrow (C\rightarrow D) \in \Delta$. Then $ \vdash \bar{A_{n}}\rightarrow (B\rightarrow C) $ and $ \vdash\bar{A_{n}}\rightarrow (C\rightarrow D)$. So, $\vdash \bar{A_{n}}\rightarrow (B\rightarrow D)$.  By definition of the empty theory, $\bar{A_{n}}\rightarrow (B\rightarrow D) \in \Delta  $.

\noindent Let $ \vdash A\rightarrow B $ and $ A \in \Delta $. Then  $ \vdash A $, so $ \vdash B $. By definition of empty theory,  $ B \in \Delta $.

\noindent Finally, it is trivial that ${\sf F_{[\rightarrow]}}$ is contained in $ \Delta $.
\end{proof}

Note that this proof is quite general. We will not repeat it for other fragments we will treat.

\begin{lem}
\label{g}
If $ \Sigma\nvdash_{{\sf F_{[\rightarrow]}} }\!A$, then there is a $ \Pi\supseteq \Sigma $  such that $ \Pi $ is a theory and $ A \notin \Pi $.
\end{lem}
\begin{proof} By assumption and by definition of provability we conclude that  $ A \notin \Sigma $. Enumerate all formulas, with infinitely many repetitions: $ B_{0}, B_{1}, ... $ and define
\begin{enumerate}
\item[] $ \Pi_{0}=\Sigma \cup F_{[\rightarrow]}$,

\item[] $ \Pi_{n+1}= \Pi_{n} \cup \left\lbrace B_{n} \right\rbrace  $ if $ \Pi_{n} , B_{n} \nvdash_{{\sf F_{[\rightarrow]}} }\!A $,

\item[]   $ \Pi_{n+1}=  \Pi_{n}  $ otherwise.
 \end{enumerate}
 Take $\Pi $ to be the union of all $ \Pi_{n} $.
By the definition of $ \Pi $, it is clear that $ A \notin \Pi $.

We must show that $\Pi $ is a theory.

\noindent Assume  $\bar{A_{i}}\rightarrow (B\rightarrow C) \in \Pi $ and $\bar{A_{i}}\rightarrow (C\rightarrow D) \in \Pi $, we must show that $\bar{A_{i}}\rightarrow (B\rightarrow D) \in \Pi $.  Let
$\bar{A_{i}}\rightarrow (B\rightarrow D) = B_{n} \notin \Pi $.

So,
$\Pi_{n}, \bar{A_{i}}\rightarrow (B\rightarrow D) \vdash_{{\sf F_{[\rightarrow]}} }\!A,$
and so,
$\Pi_{n}, \bar{A_{i}}\rightarrow (B\rightarrow C), \bar{A_{i}}\rightarrow (C\rightarrow D) \vdash_{{\sf F_{[\rightarrow]}} }\!A.$
This is a contradiction. This kind of trivial step may be skipped in later proofs.

\noindent Now let $\vdash_{{\sf F_{[\rightarrow]}} }\!C\rightarrow D$ and $ C \in \Pi $ we must show that $ D \in \Pi$. This is obvious by the same mechanism .
\end{proof}

Also this proof is quite general.

\begin{dfn}
The \textbf{Canonical Model} $ \mathfrak{M}_{{\sf F_{[\rightarrow]}}}=\langle W_{{\sf F_{[\rightarrow]}}}, \Delta, R, \Vdash\rangle $ of $\mathsf{F_{[\rightarrow ]}}$ is defined by:
\begin{enumerate}
\item  $ \Delta $ is the empty theory,

\item  $ W_{{\mathsf F_{[\rightarrow]}}} $ is the set of all theories,

\item The canonical valuation is defined by
$ \Gamma\Vdash p$ iff $p \in \Gamma .$
\end{enumerate}
\end{dfn}

\begin{lem}\label{D1}
\em\textbf{(Truth lemma)} For each $ \Gamma \in  W_{{\sf F_{[\rightarrow]}}}$ and for every formula $  C $, $$ \Gamma\Vdash   C~{\rm iff}~ C \in \Gamma .$$
\end{lem}

\begin{proof}

By induction on $C$. The base case is trivial. For the induction step, it follows from the definition of $R$ that, if $A\rightarrow B \in    \Gamma$, then $  \Gamma\Vdash A\rightarrow B $. Conversely, suppose $  \Gamma\Vdash A\rightarrow B $ and  $ \Sigma = \left\lbrace  E ~|~A\rightarrow E \in \Gamma\right\rbrace $. First, we show that $ \Sigma \in W_{{\sf F_{[\rightarrow]}}}$. So we check the conditions of Definition \ref{r1}:

(a) Let $\bar{A_{n}}\rightarrow (B\rightarrow C) \in \Sigma$ and  $\bar{A_{n}}\rightarrow (C\rightarrow D) \in \Sigma$. Then $A\rightarrow (\bar{A_{n}}\rightarrow (B\rightarrow C) )\in \Gamma$ and  $A\rightarrow (\bar{A_{n}}\rightarrow (C\rightarrow D) )\in \Gamma$. By assumption, $ \Gamma $ is a theory, so $A\rightarrow (\bar{A_{n}}\rightarrow (B\rightarrow D) )\in \Gamma$. By definition,  $\bar{A_{n}}\rightarrow (B\rightarrow D) \in \Sigma$.

(b) Let $\vdash C\rightarrow  D$ and $ C \in \Sigma$. Then by definition of theory and  definition of $  \Sigma $, we have $C\rightarrow D \in \Gamma $ and $A\rightarrow C \in \Gamma $. Again and by clause (a) of Definition \ref{r1}, we have $A\rightarrow D \in \Gamma $. Hence,  by the definition of $  \Sigma $, we conclude $ D \in \Sigma$.

(c) Assume $ C\in  {\sf F_{[\rightarrow]}}$, then it is obvious that $ A\rightarrow C\in  {\sf F_{[\rightarrow]}}$. Then $ A\rightarrow C\in  \Gamma$ and by definition of $ \Sigma $, we conclude $ C\in \Sigma$.

 So,  $ \Sigma \in W_{{\sf F_{[\rightarrow]}}}$. In order to show that $  \Gamma R\,  \Sigma $, assume $ E\rightarrow F \in \Gamma $ and $ E \in \Sigma $. By definition of $ \Sigma $, we conclude $ A \rightarrow E \in \Gamma $. We know $ \Gamma $ is a theory, so by clause (a) of Definition \ref{r1}, we conclude that $ A \rightarrow F \in \Gamma $. Hence, $ F \in \Sigma $, that is $  \Gamma R  \Sigma $.
 We have also $ A \in \Sigma$, and hence $ \Sigma \Vdash A$ by the induction hypothesis. Now we have  $  \Gamma\Vdash A\rightarrow B $, $  \Gamma R  \Sigma $ and  $ \Sigma \Vdash A$. Thus  $ \Sigma \Vdash B$ holds, and by the induction hypothesis $ B \in \Sigma $, that is, $A\rightarrow B \in    \Gamma$.
\end{proof}

\begin{thm}\label{FF}
\em\textbf{(Completeness and Soundness Theorem)} $ \Sigma\vdash_{{\sf F_{[\rightarrow ]}} }\!A $ if and only if $ \Sigma\Vdash_{{\sf F_{[\rightarrow ]}}}\!A $.
\end{thm}
\begin{proof} Soundness, i.e.\ left to right is easy. For the other direction let $ \Sigma\nvdash_{{\sf F_{[\rightarrow ]}}}\!A $. By Lemma \ref{g} there is a theory $ \Pi\supseteq\Sigma $ such that $ A \notin \Pi $. So, in the canonical model, $  \mathfrak{M}_{{\sf F_{[\rightarrow]}}}, \Pi\Vdash_{{\sf F_{[\rightarrow ]}} }\!\Sigma $ and $  \mathfrak{M}_{{\sf F_{[\rightarrow]}}},  \Pi\nVdash_{{\sf F_{[\rightarrow ]}} }\!A $, since $ A \notin \Pi $. So, $ \Sigma\nVdash_{{\sf F_{[\rightarrow ]}} }\!A$.
\end{proof}

Again this proof is quite general.

\begin{prop}\label{F12}
A formula $ A \in \mathcal{L}_{[\rightarrow]} $ is provable in $ {\sf F_{[\rightarrow]}}$ iff it is provable in ${\sf F}  $.
\end{prop}
\begin{proof}
The proof from left to right is obvious, because in ${\sf  F}  $ we have the axioms and rules of ${\sf F_{[\rightarrow]}}$. For the other direction, assume $ A $ is not provable in ${\sf F_{[\rightarrow]}}$. So, in the canonical model of $ {\sf F_{[\rightarrow]}}$ we have $ A\notin \Delta $. On the other hand it is easy to see that the canonical model of $ {\sf F_{[\rightarrow]}}$  is a Kripke model for ${\sf F}  $. Hence we can conclude that $ A $ is not provable in ${\sf F}  $.
\end{proof}

And this final proof is quite general as well. We need the next rules for {\sf F} in the following.

\begin{lem}\label{4}
The following rules hold in the logic $ {\sf F_{[\rightarrow]}}  $ (applied when there are no assumptions):
\begin{enumerate}
\item[(a)]$  \dfrac{A\rightarrow B}{(B\rightarrow C)\rightarrow (A\rightarrow C)}$,
\item[(b)]$ \dfrac{A\rightarrow B}{(C\rightarrow A)\rightarrow (C\rightarrow B)}$.
\end{enumerate}
\end{lem}
\begin{proof} We provide the proof for (a); the proof for (b) is similar.
\begin{enumerate}
\item $A\rightarrow B  $ ~~~~~~~~~~Assumption as a theorem
\item $ (B\rightarrow C)\rightarrow (A\rightarrow B ) $ ~~By 1 and rule 3
\item $  (B\rightarrow C)\rightarrow (B\rightarrow C)$ ~~Axiom 1

\item $ (B\rightarrow C)\rightarrow (A\rightarrow C) $ ~~By 2, 3 and rule 4
\end{enumerate}
\end{proof}

\subsection{The implication-conjunction fragment of {\sf F}}
In this subsection we prove strong completeness for a system $ {\sf F_{[\rightarrow, \wedge]}}  $ that characterizes the implication-conjunction fragment of {\sf F}. Frames and models for this fragment are the obvious extension of the ones in the previous subsection with the usual clauses for $\wedge$ added as are the definitions of validity and valid consequence.







\begin{thm} The axiomatization of $ {\sf F_{[\rightarrow, \wedge]}}  $ is obtained by adding to the axioms and rules 1-3 (see Thm.\ \ref{CFrr}) in $\mathcal{L}_{\rightarrow,\wedge}$ the axioms and rules
\begin{enumerate}
\item[5. ] $A\wedge B \rightarrow A$
\item[6. ] $ A\wedge B \rightarrow B$
\item[7. ] $\dfrac{A,B}{A\wedge B}$
\item[8. ] $(A\rightarrow B)\wedge(A\rightarrow C)\rightarrow(A\rightarrow B\wedge C)$
\item[9. ] $(B\rightarrow C)\wedge(C\rightarrow D)\rightarrow(B\rightarrow D)$
\end{enumerate}
\end{thm}

We see that the introduction of conjunction allows us to do without an infinite rule in the implication-conjunction fragment. This illustrates the fact that this fragment is often easier to describe than the purely implicational fragment.

\noindent The definitions of $\Gamma\vdash_{{\sf F_{[\rightarrow, \wedge]}} }\!A$ and $\Gamma\Vdash_{{\sf F_{[\rightarrow, \wedge]}} }\! A$ are straightforward (the new rule 7 has no restrictions).The ${\sf F}_{[\rightarrow, \wedge]}$-theories are then defined in the obvious way, as is the empty theory.
The definition of $R$ (Def.\ \ref{r1}) for theories is simply extended to  ${[\rightarrow, \wedge]}$-theories, and as in the previous subsection we have that the empty theory is a theory.







\begin{thm}
\em\textbf{(Weak Deduction Theorem)} $A\vdash_{{\sf F_{[\rightarrow, \wedge]}} }\!B  $ if and only if $\vdash_{{\sf F_{[\rightarrow, \wedge]}} }\! A\rightarrow B $.
\end{thm}
\begin{proof}
We just need to consider the induction step for rule 7 and for $M\!P$. If rule 7 is applied to obtain $B\wedge C$ from $B$ and $C$, we have, by the induction hypthesis,  $\vdash_{{\sf F_{[\rightarrow, \wedge]}} }\!A\rightarrow B  $ and  $\vdash_{{\sf F_{[\rightarrow, \wedge]}} }\!A\rightarrow C  $. Applying rule 7 we then get  $\vdash_{{\sf F_{[\rightarrow, \wedge]}} }\!(A\rightarrow B)\wedge(A\rightarrow C)  $. By axiom 8 and $M\!P$ we then obtain the conclusion $\vdash_{{\sf F_{[\rightarrow, \wedge]}} }\!A\rightarrow B\wedge C  $.

If $M\!P$ has been applied to obtain $A\vdash C$ from $A\vdash B$ and $\vdash B\rightarrow C$, then we have, by induction,
$\vdash A\rightarrow B$, whence by rule 7, $\vdash (A\rightarrow B)\wedge (B\rightarrow C)$. By axiom 9, and $M\!P$ we then obtain $\vdash A\rightarrow C$.
\end{proof}
  


We have the same lemma and definition of canonical model as in the previous section.

\begin{lem}
\label{g1}
 If $ \Sigma\nvdash_{{\sf F_{[\rightarrow, \wedge]}} }\!A$, then there is a $ \Pi\supseteq \Sigma $  such that $ \Pi $ is a $ {[\rightarrow, \wedge]} $-theory and  $ A \notin \Pi $.
\end{lem}

\begin{dfn}The definition of the \textbf{Canonical Model} $  \mathfrak{M}_{{\sf F_{[\rightarrow, \wedge]}}} $ is  as that of $ \mathfrak{M}_{{\sf F_{[\rightarrow]}}}$, the only difference being that instead of $ {[\rightarrow]} $-theory, we have ${[\rightarrow, \wedge]} $-theory.


\end{dfn}

\begin{lem}\label{D2}
\em\textbf{(Truth lemma)} For each $ \Gamma \in  W_{{\sf F_{[\rightarrow]}}}$ and for every formula $  C $, $$ \Gamma\Vdash   C~{\rm iff}~ C \in \Gamma .$$
\end{lem}

\begin{proof}
By induction on $C$.
We first treat conjunction.

$ (C:=A\wedge B) $ Let $ \Gamma \in  W_{{\sf F_{  [\rightarrow , \wedge]}}}$ and $  \Gamma\Vdash A\wedge B  $ then $  \Gamma\Vdash A  $ and $  \Gamma\Vdash B  $. By the induction hypothesis $ A \in \Gamma $ and $ B \in \Gamma $. $ \Gamma $ is a ${[\rightarrow, \wedge]}$-theory, so $ A\wedge B \in \Gamma $.

Now let $ A\wedge B \in \Gamma $. We have $\vdash_{{\sf F_{[\rightarrow , \wedge]}} }\!A\wedge B\rightarrow A $ and $ \vdash_{{\sf F_{[\rightarrow , \wedge]}} }\!A\wedge B\rightarrow B$, hence by definition  of ${[\rightarrow, \wedge]}$-theory we conclude that  $ A \in \Gamma $ and $ B \in \Gamma $. By the induction hypothesis $  \Gamma\Vdash A  $ and $  \Gamma\Vdash B  $ so $  \Gamma\Vdash A\wedge B  $.

For implication we just have to check the converse direction of the induction.
Again we define $  \Gamma\Vdash A\rightarrow B $, $ \Sigma = \left\lbrace  E ~|~A\rightarrow E \in \Gamma\right\rbrace $, and we see that the proof is easier than in the purely implicational case, case (a) is no longer necessary.
\end{proof}






\begin{thm}
\em\textbf{(Completeness and soundness Theorem)} $ \Sigma\vdash_{{\sf F_{[\rightarrow , \wedge]}} } A $ if and only if $ \Sigma\Vdash_{{\sf F_{[\rightarrow , \wedge]}} } A $.
\end{thm}

\begin{prop}
A formula $ A \in \mathcal{L}_{[\rightarrow, \wedge]} $ is provable in $ {\sf F_{[\rightarrow, \wedge]}}$ iff it is provable in ${\sf F}  $.
                                                  \end{prop}

\section{Fragments of extensions  of {\sf F}}\label{extf}

In this section we characterize the implicational fragments of $\mathsf{FT, FP, FP_T, FR. FRT, FRP}$.

\subsection{The implicational fragment of $ {\sf FT}  $}

We consider the following axiom schema:
$$T:~~~(A\rightarrow B)\rightarrow ((B\rightarrow C)\rightarrow (A\rightarrow C))      $$
This axiom has been proved to be strongly complete for transitive frames (see~\cite{DF2, Co}).
In the following ${\sf F}_{[\rightarrow]}+T$  means that in $  T $ only implicational formulas are substituted.
We just follow the proof of the completeness of $ {\sf F_{[\rightarrow]}} $ and the definition of the canonical model seeing that this model becomes transitive automatically. We can then claim


\begin{thm}
${\sf FT}_{[\rightarrow]}={\sf F}_{[\rightarrow]}+T$.
\end{thm}
\subsection{The implicational fragment of $ {\sf FP}  $}



We examine the following axiom schema:

$$ P:~~~ p\rightarrow (\top\rightarrow p)~~~ (p ~a ~propositional~ letter)$$


This axiom has been shown to be strongly complete for persistent models (see~\cite{DF2}). Moreover, it is sufficient to establish the strong completeness of ${\mathsf F_{[\rightarrow]}+ P}$ for $\rightarrow$-models by the method used in the previous section. Similarly, we can then claim

\begin{thm}
$\mathsf{FP_{[\rightarrow]}}=\mathsf{F_{[\rightarrow]}}+ P$.
\end{thm}

\subsection{The implicational fragment of $ {\sf FTP}  $}

K. Kikuchi, in \cite{Kent}, introduced a system characterizing the implicational fragment of {\sf BPC} as below ({\sf FTP}~=~{\sf BPC}), and also proved its simple completeness. 

$ A\rightarrow A ~~~~~~~~~~~~~~~~~A\rightarrow (B\rightarrow A) $

$( \bar{A_{n}} \rightarrow (B\rightarrow D ))\rightarrow ((\bar{A_{n}}\rightarrow (C\rightarrow B))\rightarrow (\bar{A_{n}}\rightarrow (C\rightarrow D )))$

$ \frac{A ~~~ A\rightarrow B}{B} $

\noindent His system includes the axiom 

$$ P_{T}:~~~ A\rightarrow (B\rightarrow A)$$

This axiom has been proved to be strongly complete for persistent, transitive frames (see~\cite{DF2}) and is therefore equivalent to $T + P$. In the present paper, we replace Kikuchi's complex axiom by Rules 3 and 4 of $\mathsf{F_\rightarrow}$.


Also in this case the axiom $P_T$ is sufficient to prove the strong completeness of ${\mathsf F_{[\rightarrow]}}+P_T$ by the methods used in the previous sections for the corresponding $\rightarrow$-frames. Similarly, we can then claim

\begin{thm}
$\mathsf{ FP_{T}}_{[\rightarrow]}={\sf F}_{[\rightarrow]}+ P_{T}$.
\end{thm}


\subsection{The implicational fragment of $ {\sf FR}  $}

In this subsection we will characterize the implicational fragment of {\sf FR}. The procedure is in this case somewhat more complicated to make sure that the canonical model will be reflexive.

The logic {\sf FR} obtained by adding the axiom $A\wedge(A\rightarrow B)\rightarrow B$ to {\sf F} is known to be complete w.r.t.\ reflexive frames. We will add the following rule $ I_{Refl} $ to $\mathsf{F_{[\rightarrow]}}$  to obtain an axiomatization of ${\mathsf{FR_{[\rightarrow]}}}$.

$$ I_{Refl}: \dfrac{\bar{A_{n}}\rightarrow B~~~~~~~ \bar{A_{n}}\rightarrow (B\rightarrow C)}{\bar{A_{n}}\rightarrow C}.~~~~~~~~(n\geq 0)$$

In the definition of $\vdash_\mathsf{FR_{[\rightarrow]}}$, $ I_{Refl} $ can be applied without restrictions, and to define $\mathsf{ FR}_{[\rightarrow]}$-theories the corresponding clause is added. Note that full $M\!P$ is included in $ I_{Refl} $ by $n=0$. The  canonical models that  arise will be reflexive by the unrestricted use of  $M\!P$.





\begin{thm}
\em\textbf{(Weak Deduction Theorem)} $A\vdash_{{\sf FR_{[\rightarrow]}} }\!B  $ if and only if $\vdash_{{\sf FR_{[\rightarrow]}} }\! A\rightarrow B $.
\end{thm}
\begin{proof} By Theorem \ref{lli}, we just need to consider the induction step of $ I_{Refl} $ in the nontrivial direction.
If $ A\vdash_\mathsf{ FR_{[\rightarrow]}}\!\bar{D}_n\rightarrow \!B $
and $  A\vdash_\mathsf{ FR_{[\rightarrow]}}\!\bar{D}_n\rightarrow\!(B\rightarrow C)  $. 
Then by induction hypothesis, $\vdash_\mathsf{ FR_{[\rightarrow]}}\!A\rightarrow (\bar{D}_n\rightarrow\!B)$
and $\vdash_{{\sf FR_{[\rightarrow]}} }\!A\rightarrow (\bar{D}_n\rightarrow\!(B\rightarrow C)) $, and by $ I_{Refl} $ we have $\vdash_{{\sf FR_{[\rightarrow]}} }\!A\rightarrow (\bar{D}_n\rightarrow\!
C)$.  
\end{proof}

As in the previous section, we conclude to

\begin{prop}
$ \Delta $ is an $\mathsf{FR_{[\rightarrow]}}$-theory $ \Longleftrightarrow $ $ \Delta\vdash_{{\sf FR_{[\rightarrow]}} }\!A $ if and only if $ A \in \Delta$.
\end{prop}

\begin{lem}
\label{g11}
 If $ \Sigma\nvdash_{{\sf FR_{[\rightarrow]}} }\!A$, then there is a $ \Pi\supseteq \Sigma $  such that $ \Pi $ is an $\mathsf{FR}_{[\rightarrow]} $-theory and  $ A \notin \Pi $.
\end{lem}

\begin{lem}\label{D}
\em\textbf{(Truth lemma)} For each $ \Gamma \in  W_{{\sf FR_{[\rightarrow]}}}$ and for every formula $  C $, $$ \Gamma\Vdash   C~{\rm iff}~ C \in \Gamma .$$
\end{lem}

\begin{proof}
By induction on $C$.
The proof runs the same as before  until we show for the converse direction that $ \Sigma = \left\lbrace  E ~|~A\rightarrow E \in\Gamma\right\rbrace\in W_{{\sf FR_{[\rightarrow]}}}$ and reach the point
that we assume that $ B\rightarrow C \in \Sigma$ and $ B \in \Sigma$ to prove that $C\in\Sigma$.
In the present case, by definition of $  \Sigma $, we have $A\rightarrow (B\rightarrow C) \in \Gamma $ and $A\rightarrow B \in \Gamma $. Then, by $ I_{Refl} $, $A\rightarrow C \in \Gamma $, and by definition of $  \Sigma $, $C \in  \Sigma $. So,  $ \Sigma \in W_{{\sf FR_{[\rightarrow]}}}$.

The next part of the proof is showing that $  \Gamma R\,  \Sigma $. Assume $ E\rightarrow F \in \Gamma $ and $ E \in \Sigma $. By definition of $ \Sigma $, $ A \rightarrow E \in \Gamma $. As $ \Gamma $ is an $\mathsf{FR_{[\rightarrow]}} $-theory, we have also $ A\rightarrow F \in \Gamma$, hence $F\in\Sigma$, and $ \Sigma \Vdash F$ by the induction hypothesis. 

The final conclusion is the same: we have  $  \Gamma\Vdash A\rightarrow B $, $  \Gamma R \, \Sigma $ and  $ \Sigma \Vdash A$. Thus  $ \Sigma \Vdash B$ holds, and by the induction hypothesis $ B \in \Sigma $, that is, $A\rightarrow B \in    \Gamma$.
\end{proof}



\begin{thm}
\em\textbf{(Completeness and Soundness Theorem)} $ \Sigma\vdash_{{\sf FR_{[\rightarrow ]}} } A $ if and only if $ \Sigma\vDash_{{\sf FR_{[\rightarrow]}} } A $.
\end{thm}

\begin{prop}
A formula $ A \in \mathcal{L}_{[\rightarrow]} $ is provable in $ {\sf FR_{[\rightarrow]}}$ iff it is provable in ${\sf FR}  $.
\end{prop}

\subsection{The implicational fragment of $ {\sf FRT}  $ and $ {\sf FRP}  $}

Similarly to the arguments presented in the previous subsections, the following two theorems can easily be derived:

\begin{thm}
${\sf FRP}_{[\rightarrow]}={\sf FR}_{[\rightarrow]}+  P $.
\end{thm}

\begin{thm}
${\sf FRT}_{[\rightarrow]}={\sf FR}_{[\rightarrow]}+ T $.
\end{thm}

\section{Fragments of the subintuitionistic Logic {\sf WF}}\label{fwf}

We now switch to the weaker logics with neighborhood models, starting with the basic logic {\sf WF}, discussing again both the $[\rightarrow]$- and the  $[\rightarrow,\wedge]$-fragment.

\subsection{The implicational fragment of {\sf WF} }

A \textbf{$[ \rightarrow ]$-Neighborhood Model} of subintuitionistic logic is simply the restriction to $ \rightarrow $-formulas of a Neighborhood model of {\sf WF} as introduced in~\cite{FD} and sketched in Appendix B.

\begin{thm}\label{wf}
The following formulas and rules in $ \mathcal{L} _{\rightarrow} $ axiomatize the fragment $ {\sf WF_{[\rightarrow]}} $:
\begin{enumerate}
\item $  A\rightarrow A$

\item $  \frac{A ~~A\rightarrow B}{B}$

\item $   \frac{A}{B\rightarrow A}$

\item $  \frac{A\rightarrow B~~B\rightarrow C}{A\rightarrow C}$

\item $ \frac{A\rightarrow B~~B \rightarrow A~~C\rightarrow D~~D\rightarrow C}{(A\rightarrow C)\rightarrow (B\rightarrow D)} $

\end{enumerate}
\end{thm}

The proof runs very much like the proof for $\mathsf{F_{[\rightarrow]}}$. We will just indicate the points where there may be doubts.

 \begin{dfn}\label{wf1}
$\Gamma\vdash_{{\sf WF_{[\rightarrow]}} }\! A$ iff there is a derivation of $A$ from $\Gamma$ using the rules 3, 4, 5 of Theorem \ref{wf} only when there are no assumptions, and the rule 2 of Theorem \ref{wf} only when the derivation of $A\to B$ contains no assumptions.
\end{dfn}  

In the next definition and proposition no changes are needed, and the proof of the proposition remains unchanged.
\begin{dfn}
A set of sentences $ \Delta$ is a \textbf{$\mathsf{WF_{[ \rightarrow ]}}$-theory} if and only if
\begin{enumerate}
\item  $ \vdash A\rightarrow B ~\Rightarrow~( {\rm if} ~ A \in \Delta$, then  $ B \in \Delta), $

\item  If  $ ~\vdash A~~~\Rightarrow~A \in \Delta $.
\end{enumerate}
 \end{dfn}

\begin{prop}\label{a}
$ \Delta $ is a $[ \rightarrow ]$-theory $ \Longleftrightarrow $ $ \Delta\vdash A $ if and only if $ A \in \Delta $.
\end{prop}




The proof of the weak deduction theorem also needs no changes. Just note that any rule that only applies to theorems cannot influence the validity of this theorem. The proof of the following Lindenbaum-like lemma needs no repetition either. 

\begin{thm}\label{b}
\rm{\textbf{(Weak Deduction Theorem)}} $A\vdash B  $ if and only if $ \vdash A\rightarrow B $.
\end{thm}






\begin{thm}
\label{1c}
If $ \Sigma \nvdash D$ then there is a $[ \rightarrow ]$-theory $ \Delta $ such that $\Delta \supseteq \Sigma, ~D \notin \Delta$.
\end{thm}

\begin{dfn}
Let  $ W_{{\sf WF_{[\rightarrow]}}}$ be the set of all  $ [\rightarrow] $-theories. Given a formula $ A $, we define the set $|| A ||$ as follows,
$$ ||A|| = \left\lbrace  \Delta~|~\Delta \in W_{{\sf WF_{[\rightarrow]}}}, ~ A \in \Delta\right\rbrace  .$$
\end{dfn}

\begin{lem}
\label{1kk}
Let C and D are formulas. Then
$$ || C|| \subseteq  || D|| ~~{\rm iff}~~\vdash C\rightarrow D .$$

\end{lem}
\begin{proof}  Let $ \nvdash C\rightarrow D $. Then by the Weak Deduction Theorem $ C\nvdash D $. Let $ \Sigma= \left\lbrace C\right\rbrace  $, then by Theorem \ref{1c},  there exist a  $[ \rightarrow ]$-theory $ \Gamma $ such that, $ C \in \Gamma $ and $ D \notin \Gamma $. That is $  || C|| \nsubseteq  || D||  $.

Now let $ \vdash C\rightarrow D $. Assume $ \Gamma \in W_{{\sf WF_{[\rightarrow]}}}, ~C\in \Gamma $. Then by definition of $ [\rightarrow] $-theory $ D \in \Gamma $. So $  || C || \subseteq  || D||$.
\end{proof}

Next we define the canonical model for ${\sf WF_{[\rightarrow]}}$.

\begin{dfn}
The \textbf{Canonical model} $\mathfrak{M}^{{\sf WF_{[\rightarrow]}}}=\langle W_{{\sf WF_{[\rightarrow]}}} , g, N\!B_{{\sf WF_{[\rightarrow]}}}, V\rangle$ of ${\sf WF_{[\rightarrow]}}$ is defined by:
\begin{enumerate}
\item $g$ is the set of theorems of ${\sf WF_{[\rightarrow]} }$,
\item For each $ \Gamma \in W_{{\sf WF_{[\rightarrow]}}} $ and all formulas $ A $ and $ B $, if $A\rightarrow B \in \Gamma  $ then $( ||  A||  , ||  B ||  ) {\in} N\!B_{{\sf WF_{[\rightarrow]}}} (\Gamma )$,
\item  For each $ \Gamma \in W_{{\sf WF_{[\rightarrow]}}} $  and $ X, Y \in \mathcal{P}(W_{{\sf WF_{[\rightarrow]}}})  $, if $ X\subseteq Y $ then $(X, Y ) \in  N\!B_{{\sf WF_{[\rightarrow]}}} (\Gamma )$,
\item If $ p\in P$, then $ V(p)= ||  p||   =\left\lbrace \Gamma ~|~\Gamma \in  W_{{\sf WF_{[\rightarrow]}}} ~and~ p \in \Gamma\right\rbrace  .$
\end{enumerate}
\end{dfn}
It is easy to see that in the canonical model $g={\sf WF_{[\rightarrow]}} $ is omniscient. We now use a slightly more convenient way than in~\cite{FD} and later articles to reach a Truth Lemma.

\begin{lem}
\label{xxxx}
$( ||  A||  , ||  B ||  ) {\in} N\!B_{{\sf WF_{[\rightarrow]}}} (\Gamma )$ iff $ A \rightarrow B \in\Gamma$.
\end{lem}
\begin{proof} From right to left follows immediately from the definition of canonical model. For the other direction assume $( || A||  ,  ||  B||   ) \in N\!B_{{\sf WF_{[\rightarrow]}}} (\Gamma )$. If $ ||  A||  \subseteq  ||  B||   $, then by Lemma \ref{1kk}, $ \vdash A\rightarrow B $ and so $ A\rightarrow B \in \Gamma  $.
\end{proof}

\begin{lem}\label{uuu}
\em\textbf{(Truth lemma)} For each $ \Gamma \in  W_{{\sf WF_{[\rightarrow]}}}$ and for every formula $  E $, $$ \Gamma\Vdash   E~{\rm iff}~ E \in \Gamma .$$
\end{lem}
\begin{proof} By induction on $E$. The base case is trivial. Let $E:=A\rightarrow B  $, then,

  ~~~~~~~~~~~~~~~~$\Gamma \Vdash  A\rightarrow B~~~~~~~~~~~~\Longleftrightarrow~ (A^{\mathfrak{M}^{{\sf WF_{[\rightarrow]}}}}, B^{\mathfrak{M}^{{\sf WF_{[\rightarrow]}}}} ) \in N\!B_{{\sf WF_{[\rightarrow]}}}(\Gamma) $

~~~~~~(by induction hypothesis) $~~\,\Longleftrightarrow~ (||A|| ,  || B||) \in N\!B_{{\sf WF_{[\rightarrow]}}}(\Gamma)~$

~~~~~~~~~~~~ (by Lemma~\ref{xxxx}) $~~~~~~\,\Longleftrightarrow~A\rightarrow B \in \Gamma .$
  \end{proof}

\begin{thm}
\em\textbf{(Completeness and Soundness Theorem)} $ \Sigma\vdash_{{\sf WF_{[\rightarrow ]}} }\!A $ if and only if $ \Sigma\Vdash_{{\sf WF_{[\rightarrow ]}}}\!A $.
\end{thm}

And as in the case of {\sf F} we get

\begin{prop}
A formula $  A \in \mathcal{L}_{[\rightarrow]}  $ is provable in ${\sf WF_{[\rightarrow]}} $ iff it is provable in ${\sf WF}  $.
\end{prop}

This is a good moment to discuss the strength of the equivalence rule (5) in the system $\mathsf{WF_{[\rightarrow]}}$. This strength is considerable; the system $\mathsf{WF_{[\rightarrow]}^-}$ using rules (1)-(4) only is very weak.
\begin{prop} $\mathsf{WF_{[\rightarrow]}^-}\vdash C$ iff $C$ is of the form $\bar{B_n}\rightarrow (A\rightarrow A)$.
\end{prop}
\begin{proof} $\Rightarrow:$ By induction on the length of the proof. For length 1 we have an axiom $A\rightarrow A$, and hence $n=0$. For Rule 3 it is immediately obvious that it keeps the given form intact. For Rules 2 and 4 note that in those rules, if $A\rightarrow B$ has the right form, then so does $B$ (except when $A=B$, but then the rules don't have any effect). From this observation it is clear that the proper form remains intact when applying Rules 2 and 4.

This extreme weakness persists when we add more connectives and axioms only. Even adding some standard rules leaves us with a weak system.

\noindent $\Leftarrow:$ Formulas of the given form can be produced by applying Rule 3 repeatedly to the Axiom 1.
\end{proof} 

One may further note that Axiom 1 and Rule 3 are consequences of the weak deduction theorem, and Rule 4 is a consequence of the weak deduction theorem plus restricted modus ponens. This means that one can see the whole system $\mathsf{WF_{[\rightarrow]}^-}$ as representing the weak deduction theorem plus weak modus ponens.

\subsection{The implication-conjunction fragment of {\sf WF}}
In this subsection we prove strong completeness for a system $ {\sf WF_{[\rightarrow, \wedge]}}  $ characterizing the implication-conjunction fragment of {\sf WF}. $[\rightarrow , \wedge]$-neighborhood frames and models are defined in the obvious manner as are truth, validity and valid consequence.

\begin{thm}
The axiomatization of $ {\sf WF_{[\rightarrow, \wedge]}}  $ is obtained by adding to the axioms and rules  of Thm.\ \ref{wf} in $\mathcal{L}_{\rightarrow,\wedge}$ the axioms and rules
\begin{enumerate}
\item[6.] $ A\wedge B\rightarrow A $
\item[7.] $ A\wedge B\rightarrow B $
\item[8.] $\dfrac{A~~B}{A\wedge B} $
\item[9.] $\dfrac{A\rightarrow B~~A\rightarrow C}{A\rightarrow B\wedge C}$
\end{enumerate}
\end{thm}

\noindent The definitions of $\Gamma\vdash_{{\sf WF_{[\rightarrow, \wedge]}} }\!A$ and $\Gamma\Vdash_{{\sf WF_{[\rightarrow, \wedge]}} }\! A$ are straightforward (the new rule 8 has no restrictions).The $\mathsf{WF_{[ \rightarrow, \wedge ]}}$-theories are then defined in the obvious way, as is the empty theory.



\begin{thm}
\textbf{(Weak Deduction Theorem)} $A\vdash_{{\sf WF_{[\rightarrow, \wedge]}} } B  $ if and only if $ \vdash_{{\sf WF_{[\rightarrow, \wedge]}} } A\rightarrow B $.
\end{thm}
\begin{proof} We just need to consider the induction step for rule 8.  
If $ A\vdash B   $ and  $A\vdash C $. By induction hypothesis $ \vdash A\rightarrow B $ and $ \vdash A\rightarrow C $, so by rule 9, $ \vdash A\rightarrow B\wedge C  $.
\end{proof}

We have the same lemmas and definition of canonical model as in the previous section.





\begin{thm}
\label{c}
If $ \Sigma \nvdash_{{\sf WF_{[\rightarrow, \wedge]}} } D$ then there is a $[\rightarrow , \wedge]$-theory $ \Delta $ such that $\Delta \supseteq \Sigma, ~D \notin \Delta$.
\end{thm}




\begin{dfn}The definition of \textbf{Canonical Model} $  \mathfrak{M}_{{\sf WF_{[\rightarrow, \wedge]}}} $ is  as that of $ \mathfrak{M}_{{\sf WF_{[\rightarrow]}}}$, the only difference being that instead of $ {[\rightarrow]} $-theory, we have ${[\rightarrow, \wedge]} $-theory.
\end{dfn}



\begin{thm}
\label{uuu}
\em\textbf{(Truth lemma)} For each $ \Gamma \in  W_{{\sf WF_{[\rightarrow, \wedge\wedge]}}}$ and for every formula $  E $, $$ \Gamma\Vdash   E~{\rm iff}~ E \in \Gamma .$$
\end{thm}
\begin{proof} By induction on $E$. We just treat conjunction.

Let $ \Gamma \in  W_{{\sf WF_{  [\rightarrow , \wedge]}}}$ and $  \Gamma\Vdash A\wedge B  $ then $  \Gamma\Vdash A  $ and $  \Gamma\Vdash B  $. By the induction hypothesis $ A \in \Gamma $ and $ B \in \Gamma $. $ \Gamma $ is a $[\rightarrow , \wedge]$-theory so $ A\wedge B \in \Gamma $.




  \end{proof}

\begin{thm}
\em\textbf{(Completeness and Soundness Theorem)} $ \Sigma\vdash_{{\sf WF_{[\rightarrow , \wedge]}} } A $ if and only if $ \Sigma\Vdash_{{\sf WF_{[\rightarrow , \wedge]}} } A $.
\end{thm}

\begin{prop}
A formula $ A \in \mathcal{L}_{[\rightarrow, \wedge]} $ is provable in $ {\sf WF_{[\rightarrow, \wedge]}}$ iff it is provable in ${\sf WF}  $.
 \end{prop}

\section{Implicational Fragments of extensions of {\sf WF}}\label{extwf}
In the Subsections \ref{WFI} to \ref{WFCD}, we characterize the implicational fragments of ${\sf WFI}$, ${\sf WF}\widehat{{\sf C}}$, ${\sf WF}\widehat{{\sf D}}$, and ${\sf WF}\widehat{{\sf C}} \widehat{{\sf D}}$. These were the extensions for which it is easily possible to determine their fragments. 
We conclude in Subsection \ref{unresolve} with a discussion of the fragments for which we have not yet been successful and make some conjectures.

\subsection{The implicational fragment of ${\sf WFI } $}\label{WFI}
In this subsection we will characterize the implicational fragment of {\sf WFI }. The procedure is in this case somewhat more complicated to make sure that the canonical model will be transitive.

The logic {\sf WFI } obtained by adding the axiom $(A\rightarrow B)\wedge (B\rightarrow C)\rightarrow (A\rightarrow C) $ to {\sf WF} is known to be complete w.r.t.\ generalized transitive neighborhood frames. We will add the following rule $ I_{Tran} $ to $\mathsf{WF_{[\rightarrow]}}$  to obtain an axiomatization of ${\mathsf{WFI_{[\rightarrow]}}}$.
$$ I_{Tran}:   \frac{A\rightarrow (B\rightarrow C)~~~~A\rightarrow (C\rightarrow D)}{A\rightarrow (B\rightarrow D)}$$ 
In the definition of $ \vdash_{{\sf WFI_{[\rightarrow]}} }\! A $, $ I_{Tran} $ can be applied when there are no assumptions. But note that the following rule can be applied without restrictions, and to define ${\sf WFI_{[\rightarrow]}}$-theories the corresponding clause is added:
$$ \dfrac{A\rightarrow B~~~~B\rightarrow C}{A\rightarrow C} $$
\begin{thm}
\textbf{(Weak Deduction Theorem)} $A\vdash_{{\sf WFI_{[\rightarrow]}} } B  $ if and only if $ \vdash_{{\sf WFI_{[\rightarrow]}} } A\rightarrow B $.
\end{thm}
\begin{proof} By Theorem \ref{b}, we just need to consider the following step. 
If $ A\vdash B\rightarrow C   $ and  $A\vdash C\rightarrow D $. By induction hypothesis $ \vdash A\rightarrow (B\rightarrow C) $ and $ \vdash A\rightarrow (C\rightarrow D) $, so by $ I_{Tran} $ , $ \vdash A\rightarrow (B\rightarrow D)$.
\end{proof}

As in the previous section, we conclude to

\begin{prop}
$ \Delta $ is a ${\sf WFI_{[\rightarrow]}}$-theory $ \Longleftrightarrow $ $ \Delta\vdash_{{\sf WFI_{[\rightarrow]}} } A $ if and only if $ A \in \Delta $.
\end{prop}

\begin{thm}
\label{2ci}
If $ \Sigma \nvdash_{{\sf WFI_{[\rightarrow]}} } D$ then there is a ${\sf WFI_{[\rightarrow]}}$-theory $ \Delta $ such that $\Delta \supseteq \Sigma, ~D \notin \Delta$.
\end{thm}

\begin{lem}
\label{2kk}
Let C and D are formulas. Then
$$ || C|| \subseteq  || D|| ~~{\rm iff}~~\vdash_{{\sf WFI_{[\rightarrow]}} } C\rightarrow D .$$
\end{lem}

The canonical model for ${\sf WFI_{[\rightarrow]}}$ is defined exactly in the same way as the canonical model for ${\sf WF_{[\rightarrow]}}$, and it can easily be shown that this canonical model is transitive.  


\begin{lem}
\label{xxxx22}
$( ||  A||  , ||  B ||  ) {\in} N\!B_{{\sf WFI_{[\rightarrow]}}} (\Gamma )$ iff $A\rightarrow B \in \Gamma  $.
\end{lem}

\begin{thm}
\textbf{(Truth Lemma)} 
For each $ \Gamma \in  W_{{\sf WFI_{[\rightarrow, \wedge]}}}$ and for every formula $  E $, $$ \Gamma\Vdash   E~{\rm iff}~ E \in \Gamma .$$
\end{thm}





\begin{thm}
\em\textbf{(Completeness and Soundness Theorem)} $ \Sigma\vdash_{{\sf WFI_{[\rightarrow ]}} } A $ if and only if $ \Sigma\vDash_{{\sf WFI_{[\rightarrow ]}} } A $.
\end{thm}

\begin{prop}
A formula $ A \in \mathcal{L}_{[\rightarrow]} $ is provable in ${\sf WFI_{[\rightarrow]}} $ iff it is provable in ${\sf WFI}  $.
\end{prop}

\subsection{The implicational fragment of ${\sf WF}\widehat{{\sf C}}  $}
The logic ${\sf WF}\widehat{{\sf C}}  $ obtained by adding the axiom $(A\rightarrow B\wedge C)\rightarrow (A\rightarrow B)\wedge (A\rightarrow C)$ to {\sf WF} is known to be complete w.r.t.\ the class of all generalized neighborhood frames that are closed under upset. We will add the following rule $ {\sf I_{L}} $ to $\mathsf{WF_{[\rightarrow]}}$  to obtain an axiomatization of ${\sf WF}\widehat{{\sf C}}_{[\rightarrow]} $.
$$ I_{L}:~~~~ \frac{A\rightarrow B}{(C\rightarrow A)\rightarrow (C\rightarrow B)}$$

\begin{lem}$\cite{DF2}$
The logic ${\sf WF}\widehat{{\sf C}}  $ equals the logic {\sf WF}  plus the rule $ I_{L} $.
\end{lem}

Note that in the following ${\sf WF}_{[\rightarrow]}+I_{L}$   means that in $ I_{L}$ only implicational formulas are substituted.
We just follow the proof of the completeness of $ {\sf WF_{[\rightarrow]}} $ and the definition of the canonical model insuring that this model is closed under upset. We can then claim

\begin{thm} \label{C}
${\sf WF}_{[\rightarrow]}+ I_{L}={\sf WF}\widehat{{\sf C}}_{[\rightarrow]}  $.
\end{thm}

\subsection{The implicational fragment of ${\sf WF}\widehat{{\sf D}}  $}
The logic ${\sf WF}\widehat{{\sf D}}  $ obtained by adding the axiom $(A\vee B\rightarrow C)\rightarrow (A\rightarrow C)\wedge (B\rightarrow C)$ to {\sf WF} is known to be complete w.r.t.\ the class of all generalized neighborhoood frames that are closed under downset. We will add the following rule $ I_{R} $ to $\mathsf{WF_{[\rightarrow]}}$  to obtain an axiomatization of ${\sf WF}\widehat{{\sf D}}_{[\rightarrow]} $.
$$ I_{R}:~~~~\frac{A\rightarrow B}{(B\rightarrow C)\rightarrow (A\rightarrow C)}$$

\begin{lem}$\cite{DF2}$
The logic ${\sf WF}\widehat{{\sf D}}  $ equals the logic {\sf WF}  plus the rule $ I_{R} $.
\end{lem}

Note that in the following ${\sf WF}_{[\rightarrow]}+I_{R}$   means that in $ {\sf I_{R}}$ only implicational formulas are substituted.
We just follow the proof of the completeness of $ {\sf WF_{[\rightarrow]}} $ and the definition of the canonical model insuring that this model is closed under downset. Similarly,  we can then claim

\begin{thm} \label{D}
${\sf WF}_{[\rightarrow]}+ I_{R}={\sf WF}\widehat{{\sf D}}_{[\rightarrow]}  $.
\end{thm}

\subsection{The implicational fragment of ${\sf WF}\widehat{{\sf C}}\widehat{{\sf D}} $}\label{WFCD}
Combining the rules $I_L$ and $I_R$ leads to
$$I_{LR}:~\dfrac{A\rightarrow B~~~~C\rightarrow D}{(B\rightarrow C)\rightarrow(A\rightarrow D)} $$

\begin{prop}\label{bh}
Over the logic  ${\sf WF}_{[\rightarrow]}  $, the rules $ I_{R} $ and $I_{L}  $ are equivalent to $ I_{LR} $.
\end{prop}
\begin{proof}
The proof is easy.
\end{proof}


\begin{thm}
$ {\sf WF}\widehat{C} \widehat{D}_{[\rightarrow]}= {\sf WF}_{[\rightarrow]} +I_{LR}  $.
\end{thm}
\begin{proof}
The proof is obvious by Theorems \ref{C}, \ref{D}  and Proposition \ref{bh}.
\end{proof}

\subsection{Unresolved Fragments and Conjectures related to ${\sf WF } $}\label{unresolve}

\textbf{Conjecture 1.} The axiomatization of $ {\sf WF}_{[\rightarrow]} $, $ {\sf WF_{N}}_{[\rightarrow]}$, ${\sf WFC}_{[\rightarrow]}$ and ${\sf WFD}_{[\rightarrow]}$ are the same.\\


Conjecture 1 is an especially annoying open question. The distinctness {\sf WF} and $\mathsf{WF_N}$ kept us busy for a while when we started our investigations of weak subintuitionistic logics~\cite{FD}, but it seems obvious now. The open question is whether this can be expressed using $\rightarrow$ only. In the corresponding $[\wedge, \rightarrow]$-fragments the answers are easy.\\

\noindent\textbf{Conjecture 2.} The following statements hold:
\begin{itemize}
    \item[(a)] The axiomatizations of ${\sf WF}\widehat{{\sf C}}\widehat{{\sf D}}_{[\rightarrow]}$ and ${\sf WF_{N_{2}}}_{[\rightarrow]}$ are the same.
    \item[(b)] The systems ${\sf WF}_{[\rightarrow]}$, ${\sf WF}\widehat{{\sf C}}_{[\rightarrow]}$, and ${\sf WFI}_{[\rightarrow]}$ are pairwise distinct in terms of their theorems. 
\end{itemize}

\noindent
We have been able to verify that the axiom $((p \rightarrow p) \rightarrow r) \rightarrow (q \rightarrow r)$ is derivable in ${\sf WF}\widehat{{\sf D}}_{[\rightarrow]}$, while it is not provable in ${\sf WF}_{[\rightarrow]}$. However, we have not yet identified a similar formula that separates ${\sf WF}\widehat{{\sf C}}_{[\rightarrow]}$ from ${\sf WF}_{[\rightarrow]}$ or ${\sf WFI}_{[\rightarrow]}$.

\medskip

\noindent
As long as no implicational formula has been identified that is valid in  ${\sf WF}\widehat{\sf C} _{[\rightarrow]}$, ${\sf WFI}_{[\rightarrow]}$, or  ${\sf WF_{N}}_{[\rightarrow]}$ but not in ${\sf WF}_{[\rightarrow]}$, these challenges represent a serious unresolved problem in this area (see Figure \ref{fig:i}).

\begin{cor}
${\sf WF}_{[\rightarrow, \wedge]}\neq {\sf WF_{N}}_{[\rightarrow, \wedge]}$.
\end{cor}
\begin{proof}
It is easy to prove that the axiom $ N_{b}$ is valid in $  {\sf WF_{N}}$, whereas it is not valid in $  {\sf WF}$.
\end{proof}

\begin{cor}
$ {\sf WF}\widehat{\sf C} \widehat{\sf D}_{[\rightarrow, \wedge]}\neq {\sf WF_{N_{2}}}_{[\rightarrow, \wedge]}$.
\end{cor}
\begin{proof}
It is easy to prove that the axiom $ C_{W}$ is valid in $  {\sf WF_{N_{2}}}$, whereas it is not valid in ${\sf WF}\widehat{\sf C} \widehat{\sf D}$.
\end{proof}


We add here also one conjecture about an extension of {\sf F}.
We refer to the axiom $ (A\rightarrow (A\rightarrow B))\rightarrow (A\rightarrow B) $ as $  R_{E} $. We proceed to prove the following theorem:
\begin{thm}
The axioms $   R $ and $  R_{E} $ are equivalent over the logic $ {\sf F} $.
\end{thm}
\begin{proof}
Left to right has been proved in \cite{Co}. For the other direction we proceed as follows:
\begin{enumerate}
\item $ (A\rightarrow (A\rightarrow B))\rightarrow (A\rightarrow B) $ ~~~~~~~~~~$ R_{E} $
\item $ (A\wedge (A\rightarrow B )\rightarrow (A\wedge (A\rightarrow B)\rightarrow B))\rightarrow (A\wedge (A\rightarrow B)\rightarrow B) $ ~~Substitution: 1
\item $A\wedge(A\rightarrow B)\rightarrow (A\rightarrow B)$ ~~Axiom 2
\item $A\wedge(A\rightarrow B)\rightarrow A$ ~~Axiom 1
\item $ (A\rightarrow B)\rightarrow (A\wedge(A\rightarrow B)\rightarrow B) $ ~~By 4 and Lemma \ref{4}
\item $A\wedge(A\rightarrow B)\rightarrow (A\wedge(A\rightarrow B)\rightarrow B)$~~~By 3 and 5
\item $A\wedge(A\rightarrow B)\rightarrow B$~~~By 2 and 6
\end{enumerate}

\end{proof}
On the basis of this fact one might be inclined to think that $\mathsf{F}+ R_{E} $ will characterize  the implicational fragment of $\mathsf{F} $. However, we have been unable to prove this and have taken our recourse to $ I_{Refl} $. We leave this as a small open problem.

\section{The implication-conjunction fragments of extensions of {\sf WF}}\label{icextwf}

In this section we characterize the implication-conjunction fragments of {\sf WFC}, ${\sf WFC_{W}}$ and ${\sf WFN_{b}}$, cases where the determination of the implicational fragments have so far been left open.

\subsection{The implication-conjunction fragment of ${\sf WFC}$}

The logic ${\sf WFC} $ obtained by adding the axiom $(A \rightarrow  B)\wedge (  A\rightarrow C ) \rightarrow (A\rightarrow B\wedge C) $ to {\sf WF} is known to be complete w.r.t.\ the class of all generalized neighborhoood frames that are closed under intersection. 
So, we consider the following axiom schema:
$$  C:~~ ~(A \rightarrow  B)\wedge (  A\rightarrow C ) \rightarrow (A\rightarrow B\wedge C) $$
where in the following ${\sf WF}_{[\rightarrow, \wedge]}+ C$ means that in $  C$ only implication-cojunction formulas are substituted.
We just follow the proof of the completeness of $ {\sf WF_{[\rightarrow, \wedge]}} $ and the definition of the canonical model insuring that this model is closed under intersection. We can then claim

\begin{thm}
${\sf WF}_{[\rightarrow, \wedge]}+ C={\sf WFC}_{[\rightarrow, \wedge]}$.
\end{thm}

\subsection{The implication-conjunction fragment of $ {\sf WFC_{W}}$}
The logic ${\sf WFC_{W}} $ obtained by adding the axiom $(A \rightarrow  B)\rightarrow (C\wedge A\rightarrow C\wedge B)   $ to {\sf WF} is known to be complete w.r.t.\ the class of all generalized neighborhoood frames that are closed under weak intersection. 
So, we consider the following axiom schema:
$$  C_{W}:~~ ~(A \rightarrow  B)\rightarrow (C\wedge A\rightarrow C\wedge B)  $$
where in the following ${\sf WF}_{[\rightarrow, \wedge]}+ C_{W}$ means that in $  C_{W}$ only implication-conjunction formulas are substituted.
We just follow the proof of the completeness of $ {\sf WF_{[\rightarrow, \wedge]}} $ and the definition of the canonical model  insuring that this model is closed under weak intersection. Similarly,  we can then claim


\begin{thm}
${\sf WF}_{[\rightarrow, \wedge]}+ C_{W}={\sf WFC_{W}}_{[\rightarrow, \wedge]}$.
\end{thm}

\subsection{The implication-conjunction fragment of $  {\sf WFN_{b}}$}
The logic $ {\sf WFN_{b}} $ obtained by adding the axiom $(A\rightarrow B)\leftrightarrow (A \rightarrow A \wedge B)$  to {\sf WF} is known to be complete w.r.t.\ the class of all generalized neighborhoood frames that satisfy the $  N_{b}$-condition. 
So, we consider the following axiom schema:
$$  N_{b}:~~~ (A\rightarrow B)\leftrightarrow (A \rightarrow A \wedge B),~~~~~~~~~~~$$
where in the following ${\sf WF}_{[\rightarrow, \wedge]}+N_{b}$ means that in $  N_{b}$ only implication-conjunction formulas are substituted.
We just follow the proof of the completeness of $ {\sf WF_{[\rightarrow, \wedge]}} $ and the definition of the canonical model insuring that this model satisfies the $N_{b}$-condtion. Similarly,  we can then claim

\begin{thm}
${\sf WF}_{[\rightarrow, \wedge]}+ N_{b}={\sf WFN_{b}}_{[\rightarrow, \wedge]}$.
\end{thm}

\section{Conclusion} In this paper we categorized principles valid for implication in {\sf IPC} by studying the $[\rightarrow]$-fragments of subintuitionistic logics, both ones with Kripke-models as well as ones with neighborhood models.

\label{references}
\

\begin{figure}[ht]
\begin{center}
$$\begin{tikzpicture}[line cap=round,line join=round,>=triangle 45,x=1cm,y=1cm]
\draw (-5.52,6.96) node[anchor=north west] {{\sf  BPC}};
\draw (-2.52,7.5) node[anchor=north west] {{\sf IPC}};
\draw (-5.34,4.9) node[anchor=north west] {{\sf  FT}};
\draw (-3.22,4) node[anchor=north west] {${\sf F}$};
\draw (-3.34,6.08) node[anchor=north west] {{\sf  FP}};
\draw (-0.42,4.96) node[anchor=north west] {{\sf  FR}};
\draw (-0.52,6.9) node[anchor=north west] {{\sf  FPR}};
\draw (-2.5,5.5) node[anchor=north west] {{\sf FTR}};
\draw [line width=1pt] (-5,6.34)-- (-4.96,4.82);
\draw [line width=1pt] (-4.74,4.36)-- (-3.3,3.76);
\draw [line width=1pt] (-3,5.5)-- (-2.96,4);
\draw [line width=1pt] (-4.56,6.3)-- (-3.4,5.86);
\draw [line width=1pt] (-2.56,5.8)-- (-0.42,6.3);
\draw [line width=1pt] (-4.5,6.64)-- (-2.8,7);
\draw [line width=1pt] (-1.56,6.86)-- (-0.64,6.62);
\draw [line width=1pt] (-0.04,6.16)-- (-0.02,4.98);
\draw [line width=1pt] (-2.58,3.78)-- (-0.36,4.42);
\draw [line width=1pt,dash pattern=on 2pt off 3pt] (-4.56,4.68)-- (-2.48,5.26);
\draw [line width=1pt,dash pattern=on 2pt off 3pt] (-1.56,5.12)-- (-0.56,4.76);
\draw [line width=1pt,dash pattern=on 2pt off 3pt] (-2.02,6.92)-- (-2,5.5);

\draw (-6.72,3.46) node[anchor=north west] {${\sf  WF{\textsf{$\widehat{\sf{C}}$}}C }$};
\draw (-3.22,4) node[anchor=north west] {{\sf F}};
\draw (-6.54,1.4) node[anchor=north west] {${\sf  WF  {\textsf{$\widehat{\sf{C}}$}}}$};
\draw (-4.34,0.5) node[anchor=north west] {{\sf WF}};
\draw (-4.44,2.58) node[anchor=north west] {{\sf  WFC}};
\draw (-1.42,1.46) node[anchor=north west] {${\sf  WF_{N}}$};
\draw (-1.52,3.4) node[anchor=north west] {${\sf  WF_{N}C}$};
\draw (-3.5,2.1) node[anchor=north west] {${\sf WF_{N_{2}}}$};
\draw [line width=1pt] (-6,2.84)-- (-5.96,1.32);
\draw [line width=1pt] (-5.74,0.86)-- (-4.3,0.26);
\draw [line width=1pt] (-4,2)-- (-3.96,0.5);
\draw [line width=1pt] (-5.56,2.8)-- (-4.4,2.36);
\draw [line width=1pt] (-3.56,2.3)-- (-1.42,2.8);
\draw [line width=1pt] (-5.5,3.14)-- (-3.8,3.5);
\draw [line width=1pt] (-2.56,3.36)-- (-1.64,3.12);
\draw [line width=1pt] (-1.04,2.66)-- (-1.02,1.48);
\draw [line width=1pt] (-3.58,0.28)-- (-1.36,0.92);
\draw [line width=1pt,dash pattern=on 2pt off 3pt] (-5.56,1.18)-- (-3.48,1.76);
\draw [line width=1pt,dash pattern=on 2pt off 3pt] (-2.56,1.62)-- (-1.56,1.26);
\draw [line width=1pt,dash pattern=on 2pt off 3pt] (-3.02,3.42)-- (-3,2);
\end{tikzpicture}$$
\end{center}
\caption{Lattice of some subintuitionistic logics  }
\label{fig:wfc}
\end{figure}

\begin{figure}[ht]
\begin{center}
$$\begin{tikzpicture}[line cap=round,line join=round,>=triangle 45,x=1cm,y=1cm]
\draw (-5.52,6.96) node[anchor=north west] {$ P_{T}$};
\draw (-2.52,7.5) node[anchor=north west] {${\sf IPC_{[\rightarrow]}}$};
\draw (-5.34,4.9) node[anchor=north west] {$ T $};
\draw (-3.22,4) node[anchor=north west] {${\sf F_{[\rightarrow]}}$};
\draw (-3.34,6.08) node[anchor=north west] {$ P $};
\draw (-0.42,4.96) node[anchor=north west] {$ I_{Refl}, R_{E} $?};
\draw (-0.52,6.9) node[anchor=north west] {$ P+I_{Refl} $};
\draw (-2.5,5.5) node[anchor=north west] {$ T+I_{Refl} $};
\draw [line width=1pt] (-5,6.34)-- (-4.96,4.82);
\draw [line width=1pt] (-4.74,4.36)-- (-3.3,3.76);
\draw [line width=1pt] (-3,5.5)-- (-2.96,4);
\draw [line width=1pt] (-4.56,6.3)-- (-3.4,5.86);
\draw [line width=1pt] (-2.56,5.8)-- (-0.42,6.3);
\draw [line width=1pt] (-4.5,6.64)-- (-2.8,7);
\draw [line width=1pt] (-1.56,6.86)-- (-0.64,6.62);
\draw [line width=1pt] (-0.04,6.16)-- (-0.02,4.98);
\draw [line width=1pt] (-2.58,3.78)-- (-0.36,4.42);
\draw [line width=1pt,dash pattern=on 2pt off 3pt] (-4.56,4.68)-- (-2.48,5.26);
\draw [line width=1pt,dash pattern=on 2pt off 3pt] (-1.56,5.12)-- (-0.56,4.76);
\draw [line width=1pt,dash pattern=on 2pt off 3pt] (-2.02,6.92)-- (-2,5.5);

\draw (-6.22,3.46) node[anchor=north west] {?};
\draw (-3.22,4) node[anchor=north west] {${\sf F_{[\rightarrow]}}$};
\draw (-6.44,1.4) node[anchor=north west] {$I_{L}$};
\draw (-4.34,0.5) node[anchor=north west] {${\sf WF_{[\rightarrow]}}$};
\draw (-4.24,2.58) node[anchor=north west] {?};
\draw (-1.42,1.46) node[anchor=north west] {?};
\draw (-1.52,3.4) node[anchor=north west] {?};
\draw (-3.22,2.1) node[anchor=north west] {?};
\draw [line width=1pt] (-6,2.84)-- (-5.96,1.32);
\draw [line width=1pt] (-5.74,0.86)-- (-4.3,0.26);
\draw [line width=1pt] (-4,2)-- (-3.96,0.5);
\draw [line width=1pt] (-5.56,2.8)-- (-4.4,2.36);
\draw [line width=1pt] (-3.56,2.3)-- (-1.42,2.8);
\draw [line width=1pt] (-5.5,3.14)-- (-3.8,3.5);
\draw [line width=1pt] (-2.56,3.36)-- (-1.64,3.12);
\draw [line width=1pt] (-1.04,2.66)-- (-1.02,1.48);
\draw [line width=1pt] (-3.58,0.28)-- (-1.36,0.92);
\draw [line width=1pt,dash pattern=on 2pt off 3pt] (-5.56,1.18)-- (-3.48,1.76);
\draw [line width=1pt,dash pattern=on 2pt off 3pt] (-2.56,1.62)-- (-1.56,1.26);
\draw [line width=1pt,dash pattern=on 2pt off 3pt] (-3.02,3.42)-- (-3,2);
\end{tikzpicture}$$
\end{center}
\scriptsize \caption{ Implicational fragments of some subintuitionistic logics.}
\vspace{-0.3cm} 
\begin{center}
\scriptsize 
$P_{T} : A \rightarrow (B \rightarrow A)$, ~ $P = p \rightarrow (T \rightarrow p)$, ~ $ T:(A\rightarrow B)\rightarrow ((B\rightarrow C)\rightarrow (A\rightarrow C))$,  ~$R_{E}: (A\rightarrow (A\rightarrow B))\rightarrow (A\rightarrow B) $, \quad  $ I_{Refl}: \dfrac{\bar{A_{n}}\rightarrow B~~~~ \bar{A_{n}}\rightarrow (B\rightarrow C)}{\bar{A_{n}}\rightarrow C}~~~(n\geq 0)$, \qquad 
$ I_{L}:~~~~ \frac{A\rightarrow B}{(C\rightarrow A)\rightarrow (C\rightarrow B)}$
\end{center}
\label{fig:i}
\end{figure}
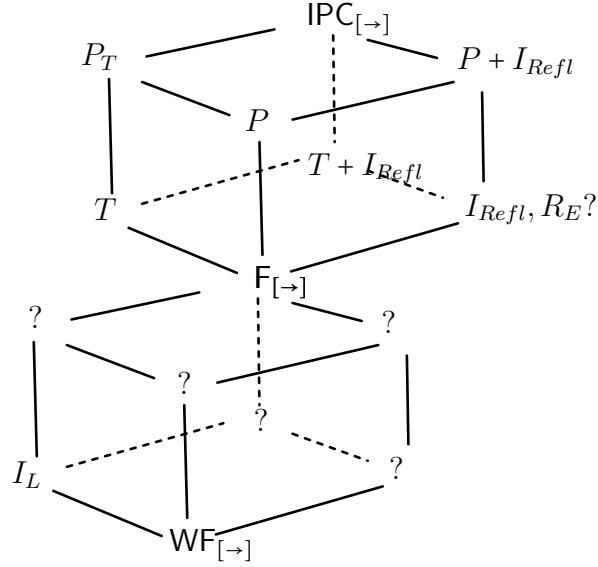
\appendix

\section{Subintuitionistic Logics with Kripke Semantics}\label{kripke}
The results in this section are stated without proofs. For detailed proofs, see \cite{Dic} and \cite{DF2}.

\begin{dfn}
A \textbf{rooted subintuitionistic Kripke frame} is a triple $ \langle W, g, R \rangle $. R is a binary relation on W, $ g \in W $, the \textbf{root}, is \textbf{omniscient}, i.e.\ $ g R w $ for each $ w \in W$. A \textbf{rooted subintuitionistic Kripke model} is a quadruple $ \langle W, g, R, V \rangle $ with $ V:P\rightarrow 2^{W} $  a valuation function on the set of propositional variables $P$. The binary \textbf{forcing} relation $ \Vdash $ is defined on $ w \in W $ as follows.
\begin{enumerate}
\item $ w \Vdash p~~~~~~~~\Leftrightarrow ~ ~w \in V(p) $, for any $ p \in P $,
\item  $w \Vdash A\wedge B  ~~\Leftrightarrow~ ~w \Vdash A  \ and \  w \Vdash B$,
\item  $w \Vdash A\vee B  ~~\Leftrightarrow~~ w \Vdash A  \ or \ w \Vdash B$,
\item  $w \Vdash A\rightarrow B  ~\Leftrightarrow$ ~ for each v with $w R v$, if $v \Vdash A$ then $v \Vdash B$.
\end{enumerate}

\noindent $\M\Vdash  A$ if for all $ w \in W ,$ $\M, w\Vdash A$. If all models force $ A $, we write $ \Vdash\! A $ and call $ A $ \textbf{valid}.
\end{dfn}
\begin{dfn}
{\sf F} is the logic given by the following axioms and rules,
\begin{enumerate}
\item $ A\rightarrow A\vee B ~~~~~~~~~~~~~~\textit{7}.~A\wedge (B\vee C)\rightarrow (A\wedge B) \vee(A\wedge C)$
\item $ B\rightarrow A\vee B ~~~~~~~~~~~~~~\textit{8}.~(A\rightarrow B) \wedge (B\rightarrow C)\rightarrow (A\rightarrow C)$
\item $ A\wedge B \rightarrow A ~~~~~~~~~~~~~~\textit{9}. ~(A\rightarrow B) \wedge (A\rightarrow C) \rightarrow(A \rightarrow B\wedge C) $
\item $ A\wedge B \rightarrow B~~~~~~~~~~~~~~\textit{10}.~A\rightarrow A$
\item $\frac{A~~B}{A\wedge B}~~~~~~~~~~~~~~~~~~~~~\textit{11}. ~(A\rightarrow C) \wedge (B\rightarrow C)\rightarrow (A\vee B\rightarrow C) $
\item $ \frac{A ~~A\rightarrow B}{B} ~~~~~~~~~~~~~~~~~\textit{12} .~ \frac{A}{B\rightarrow A}\,\,\, \mbox{(\textbf{a fortiori} or \textbf{weakening})}$
\end{enumerate}
\end{dfn}
\begin{thm} 
{\sf F} is sound and strongly complete with respect to the class of rooted subintuitionistic Kripke frames.
\end{thm}

\begin{dfn}
A rooted subintuitionistic Kripke frame $ \mathfrak{F}=\langle W, g, R \rangle  $ is \textbf{reflexive} if $ wRw $ holds for all $ w \in W$, and \textbf{transitive} if $ w Rv $ and $ vRs $ imply $ wRs $ for all $ w, v, s \in W $.
\end{dfn}

\begin{lem}\quad
\begin{enumerate}
\item[(a)]  The formula $A\wedge (A\rightarrow B)\rightarrow B $ characterizes the class of reflexive Kripke frames.
\item[(b)]   The formula $(A\rightarrow B)\rightarrow (C\rightarrow (A\rightarrow B) )$ characterizes the class of transitive Kripke frames.
\item[(c)]  The formula $(A\rightarrow B)\rightarrow ((B\rightarrow C)\rightarrow (A\rightarrow C) )$ characterizes the class of transitive
Kripke frames.
\end{enumerate}
\end{lem}

\begin{dfn}
The valuation $ V $ is called \textbf{persistent} if $  V(p) $ is upward-closed for all propositional variables
$ p, $
that is: if $w \in V(p)$ and $wRw'$, then $w'\in V(p)$.
\end{dfn}
\begin{lem}
\label{ch}
The formula $ p\rightarrow  (A\rightarrow p )$ with $ p $ a propositional variable characterizes the class of rooted subintuitionistic Kripke models $\M=\langle W, g, R, V\rangle $ in which $V$ is persistent.
\end{lem}
It is clear that $p\rightarrow  (A\rightarrow p ) $ can be replaced by $p\rightarrow  (\top\rightarrow p ) $. In the remainder of this section we will be interested in the following axiom schemas.

\begin{enumerate}
\item[] ($ R$) ~ $A\wedge (A\rightarrow B)\rightarrow B $
\item[] ($P$) ~ $p \rightarrow (\top \rightarrow p) $  ~~ ($ p $ a propositional letter)
\item[] ($ T_{1}$)  ~ $(A\rightarrow B)\rightarrow (C \rightarrow (A\rightarrow B))$
\item[] ($ T_{2}$)  ~ $(A\rightarrow B)\rightarrow ((B\rightarrow C)\rightarrow (A\rightarrow C))$
\item[] ($P_{T}$) ~ $A \rightarrow (B \rightarrow A) $
\end{enumerate}
\begin{lem} \label{FTPR}\quad
\label{mmm}
\begin{enumerate}
\item[(a)] If~ ${\sf FR}\subseteq$~\textit{L}, then the canonical model of logic \textit{L} is reflexive.
\item[(b)] If~ ${\sf FP}\subseteq$~\textit{L}, then the canonical model of logic \textit{L} is persistent.
\item[(c)] If~ ${\sf FT_{1}}\subseteq$~\textit{L}, then the canonical model of logic \textit{L} is transitive.
\item[(c)] If~ ${\sf FT_{2}}\subseteq$~\textit{L}, then the canonical model of logic \textit{L} is transitive.
\item[(d)] If~ ${\sf FP_{T}}\subseteq$~\textit{L}, then the canonical model of logic \textit{L} is transitive and persistent.
\end{enumerate}
\end{lem}

Since ${\sf FT_{1}}$ and ${\sf FT_{2}}$ did turn out to be equivalent we will refer to both by {\sf FT}.
Visser's logic {\sf BPC} is equivalent to {\sf FTP} or $\mathsf{ FP_T}$  in our terminology.

\section{Subintuitionistic Logics with Neighborhood Semantics}\label{neighbor}

The results in this section are stated without proofs. For detailed proofs, see \cite{FD} and \cite{Dic4}.
\begin{dfn}
\label{11}
A triple $ \mathfrak{F}=\langle W, g, N\!B\rangle $ is called a \textbf{ Neighborhood Frame} of subintuitionistic logic if $W$ is a non-empty set and $N\!B$ is a neighborhood function from $W$ into $ \mathcal{P}((\mathcal{P}(W))^{2} )$ such that
\begin{enumerate}
\item   $\forall w \in W , ~ \forall X, Y \in \mathcal{P}(W), ~(X\subseteq Y~\Rightarrow~(X, Y) \in N\!B(w)); $
\item $ N\!B(g)= \left\lbrace (X, Y) \in (\mathcal{P}(W))^{2}~|~X\subseteq Y\right\rbrace.$
\end{enumerate}
Here $ g $ is called \textbf{omniscient} (i.e.\ has the property 2).
\end{dfn}
We use the existence of omniscient worlds in the proofs of soundness and of characterization of properties of frames.
\begin{dfn}
A \textbf{Neighborhood Model} of subintuitionistic logic is a tuple $\mathfrak{M}= \langle W, g, N\!B, V\rangle $, where $ \langle W, g, N\!B\rangle $ is a neighborhood frame of subintuitionistic logic and $ V: P\rightarrow 2^{W} $ a valuation function on the set of propositional variables $P$.
\end{dfn}
\begin{dfn}
\label{truth}
\textbf{(Truth in a Neighborhood Model)} Let $ M=\langle W, g, N\!B,V \rangle $ be a model and $ w\in W$. Truth of a propositional formula in a world $w$ is defined inductively as follows.
\begin{enumerate}

\item ~$\mathfrak{M},w \Vdash p~~~~~~ ~~\Leftrightarrow~ w \in V(p)$;

\item ~$ \mathfrak{M},w \Vdash A\wedge B~~\Leftrightarrow~ \mathfrak{M},w \Vdash A ~{\rm and}~ \mathfrak{M},w \Vdash B$;

\item ~$ \mathfrak{M},w \Vdash A\vee B~~\Leftrightarrow ~\mathfrak{M},w \Vdash A ~{\rm or}~ \mathfrak{M},w \Vdash B$;

\item ~$ \mathfrak{M},w \Vdash A\rightarrow B~\Leftrightarrow ~  \left(  A^{\mathfrak{M}}, B^{\mathfrak{M}}\right)   \in N\!B(w)$;

\item ~$ \mathfrak{M},w \nVdash \perp,$
\end{enumerate}
where $ A^{\mathfrak{M}} $ denotes the truth set of $ A $.
A formula A is \textbf{true in a model} $\mathfrak{M}{=}\,\langle W, g, N\!B, V\rangle $, $\mathfrak{M}\,{\Vdash}\,A$ if for all $ w \in W, ~M, w \Vdash A $ and if all models force A, we write $ \Vdash A $ and call A \textbf{valid}. A formula A  is \textbf{valid on a frame}  $ \mathfrak{F}=\langle W, g, N\!B\rangle $, $ \mathfrak{F}\Vdash A $ if $ A $ is true in every model based on that frame.
\end{dfn}

\begin{dfn}
{\sf WF} is the logic given by the following axioms and rules,
\begin{enumerate}
\item $ A\rightarrow A\vee B ~~~~~~~~~~~~~~\textit{9}.~A\rightarrow A$

\item $ B\rightarrow A\vee B ~~~~~~~~~~~~~ \textit{10}.~\frac{A ~~A\rightarrow B}{B}$

\item $ A\wedge B \rightarrow A ~~~~~~~~~~~~~\textit{11}.~ \frac{A~~B}{A\wedge B}$

\item $ A\wedge B \rightarrow B~~~~~~~~~~~~~\textit{12} .~ \frac{A}{B\rightarrow A}$

\item $ \frac{A\rightarrow B~~A\rightarrow C}{A\rightarrow B\wedge C}~~~~~~~~~~~~~\textit{13} .~ \frac{A\rightarrow B~~B\rightarrow C}{A\rightarrow C}$

\item $ \frac{A\rightarrow C~~B\rightarrow C}{A\vee B \rightarrow  C}~~~~~~~~~~~~~\textit{14} .~\frac{A\leftrightarrow B~~C\leftrightarrow D}{(A\rightarrow C)\leftrightarrow (B\rightarrow D)} $

\item $A\wedge (B\vee C)\rightarrow (A\wedge B) \vee(A\wedge C)  ~~~~~~~~~~~~~~~~~~$

\item $ \perp\rightarrow A $
\end{enumerate}
\end{dfn}
The rules are to be applied in such a way that, if the formulas above the line are theorems of {\sf WF}, then the formula below the line is a theorem as well except for Rule 10 which functions as weak $M\!P$.

\begin{thm}
The logic {\sf WF} is sound and strongly complete with respect to the class of neighbourhood frames.
\end{thm}

As usual, we can extend {\sf WF} by adding various rules and axiom schemes.
In this paper, we will consider the following rules:
$$ \frac{A\rightarrow B\vee C~~~~~~C\rightarrow A\vee D~~~~~~A\wedge C\wedge  D\rightarrow B~~~~~~ A\wedge C\wedge B\rightarrow D}{(A\rightarrow B)\leftrightarrow (C\rightarrow D)}~~~~ N$$
$$~~~~~ \frac{C\rightarrow A\vee D~~~~~~~ A\wedge C\wedge B\rightarrow D}{(A\rightarrow B)\rightarrow (C\rightarrow D)}~~~~~~~~~~ N_{2}$$
and the axiom:
$$(A\rightarrow B)\wedge (B\rightarrow C)\rightarrow(A\rightarrow  C)~~~~~~~~~~I$$
$$(A\rightarrow B)\wedge (A\rightarrow C)\rightarrow(A\rightarrow B\wedge C)  ~~~~C$$
$$(A\rightarrow C)\wedge (B\rightarrow C)\rightarrow(A\vee B \rightarrow C)  ~~~~D$$
$$(A\rightarrow B\wedge C)\rightarrow (A\rightarrow B)\wedge (A\rightarrow C) ~~~~\widehat{C} $$
$$ (A\vee B \rightarrow C)\rightarrow (A\rightarrow C)\wedge (B\rightarrow C) ~~~~\widehat{D}$$
$$(A\rightarrow B)\leftrightarrow (A \rightarrow A \wedge B)~~~~~~~~~~~~~~~~~~N_{b} $$
$$ (A \rightarrow  B)\rightarrow (C\wedge A\rightarrow C\wedge B)~~~~~~~~~~~~~C_{W}$$
If $ \Gamma\subseteq\lbrace N, N_{2},   I,  C, D, \widehat{C} , \widehat{D}, N_{b}, C_{W}\rbrace $, we will write ${\sf WF\Gamma }$ for the logic obtained from {\sf WF} by adding to  {\sf WF}  the schemas in $ \Gamma $ as new axioms.

In order to give soundness and completeness results for extensions of {\sf WF} with respect to (classes of) neighbourhood models, we introduce the following definition.

\begin{dfn}
For every neighbourhood frame  $  \mathfrak{F} $, we list some relevant properties as follows (here $ X, Y, Z, X', Y' \in \mathcal{P}(W) $, $ w \in W $
and we denote the complement of $ X $ by $\overline{X} $):
\begin{enumerate}
\item[$ \bullet $] \fl \, is closed under \textbf{intersection} if and only if for all $ w \in W $, if $ (X, Y)\in N\!B(w) $, $ (X, Z)\in N\!B(w) $ then $ (X,  Y\cap Z) \in N\!B(w) $.
\item[$ \bullet $] \fl \, is closed under \textbf{union} if and only if for all $ w \in W $, if $ (X, Y)\in N\!B(w) $, $ (Z, Y)\in N\!B(w) $ then $ (X\cup Z,  Y) \in N\!B(w) $.
\item[$ \bullet $] \fl \, ps \textbf{transitive} if and only if for all $ w \in W $, if $ (X, Y)\in N\!B(w) $, $ (Y, Z)\in N\!B(w) $ then $ (X,  Z) \in N\!B(w) $,
\item[$ \bullet $] \fl \, is closed under \textbf{upset} if and only if for all $ w \in W $, if $ (X, Y)\in N\!B(w) $ and $ Y\subseteq Z $ then $ (X, Z) \in N\!B(w) $,
\item[$ \bullet $] \fl \, is closed under \textbf{downset} if and only if for all $ w \in W $, if $ (X, Y)\in N\!B(w) $ and $ Z\subseteq X $ then $ (Z, Y) \in N\!B(w) $,
\item[$ \bullet $] \fl \, is closed under \textbf{equivalence} if and only if for all $ w \in W $, if $ (X, Y)\in N\!B(w) $ and $ \overline{X} \cup Y =\overline{X^{'}} \cup Y^{'}$ then $ (X^{'}, Y^{'}) \in N\!B(w) $,
\item[$ \bullet $] \fl \, is closed under \textbf{superset equivalence} if and only if for all $ w \in W $, if $ (X, Y)\in N\!B(w) $ and $ \overline{X} \cup Y \subseteq \overline{X^{'}} \cup Y^{'}$ then $ (X^{'}, Y^{'}) \in N\!B(w) $,
\item[$ \bullet $]  $ \fl $ satisfies the $ {\sf N_{b} }$-condition if and only if for all $ w \in W $,
$(X, Y)\in N\!B(w)$  iff $ (X, X\cap Y) \in N\!B(w)$,
\item[$ \bullet $]  $ \fl $ is closed under \textbf{weak intersection} if and only if for all $ w \in W $,
If $(X, Y)\in N\!B(w)$ then for all $ Z$, $ (X \cap Z, Y\cap Z) \in N\!B(w)$.
\end{enumerate}
\end{dfn}

\begin{prop}\label{4bprop}
We have the following correspondence results between formulas and the
properties of the neighbourhood function defined above:
\begin{enumerate}
\item[$ \bullet $] Axiom $C$ corresponds to closure under intersection;
\item[$ \bullet $] Axiom $D$ corresponds to closure under union;
\item[$ \bullet $] Axiom $I$ corresponds to transitivity;
\item[$ \bullet $] Axiom $\widehat{ C}$ corresponds to closure under upset;
\item[$ \bullet $] Axiom $\widehat{D}$ corresponds to closure under downset;
\item[$ \bullet $] Axiom $N_{b}$  corresponds to satisfying the $ {\sf N_{b} }$-condition;
\item[$ \bullet $] Axiom $C_{W}$  corresponds to closure under weak intersection.
\item[$ \bullet $] Rule $N$ corresponds to closure under equivalence;
\item[$ \bullet $] Rule $ N_{2 }$ corresponds to closure under superset equivalence.
\end{enumerate}
\end{prop}

\begin{thm}\label{4b}
If $ \Gamma\subseteq\lbrace  I,  C, D, \widehat{C}, \widehat{D}, N_{b}, C_{W}, N, N_{2} \rbrace $, then  ${\sf WF\Gamma }$ is sound and strongly complete
with respect to the class of general neighborhood frames with all the
properties defined by the axioms and rules in $ \Gamma $.
\end{thm}
Finally, we draw attention to a new fact which cannot be found in our earlier papers concerning the logics $\mathsf{WF_N}$ and $\mathsf{WF_{N_2}}$ which are in first instance not defined in connection with the neighborhood properties of proposition~\ref{4bprop} but in connection with the more standard unary $N$-neghborhood-frames of modal logic (see~\cite{Dic4}).
\begin{obs} We introduce a new rule\footnote{The new rule $ N'$, denoted by $N'$ was first suggested to us by Narbe Aboolian. We thank him for this suggestion.} $N'$ as an alternative to the rule $N$:
$$ \frac{A\rightarrow B\vee C~~~~~C\rightarrow A\vee D~~~~~A\wedge  D\rightarrow B~~~~~ C\wedge B\rightarrow D}{(A\rightarrow B)\leftrightarrow (C\rightarrow D)}~~~~N'$$
It is straightforward to show that $N'$ can be derived from $N$. Furthermore, it can also be easily demonstrated that the addition of $N'$ to {\sf WF} suffices to establish the completeness of the resulting logic $\mathsf{WF_N}$ with respect to the class of general neighborhood frames that are closed under equivalence.

Similarly, we have recently identified an alternative rule, denoted by $ N'_{2} $, which can replace $N_{2} $:
$$~~~~~ \frac{C\rightarrow A\vee D~~~~~~~ C\wedge B\rightarrow D}{(A\rightarrow B)\rightarrow (C\rightarrow D)}~~~~~~~~~~N'_{2}$$
It is straightforward to demonstrate that 
$  N'_{2} $ can be derived from $ N_{2} $. Furthermore, it can also be easily demonstrated that the addition of $ N'_{2} $ to {\sf WF} suffices to establish the completeness of the resulting logic $\mathsf{WF_{N_2}}$ with respect to the class of general neighborhood frames that are closed under superset equivalence.

Thus, the new rules $N'$  and $  N'_{2} $ provide simplified ways to analyze and extend {\sf WF} while preserving key properties such as soundness and completeness.
\end{obs}

\end{document}